\documentclass[12pt, reqno]{amsart}
\usepackage{ amsmath,amsthm, amscd, amsfonts, amssymb, graphicx, color}
\usepackage[bookmarksnumbered, colorlinks, plainpages]{hyperref}
\newtheorem{thm}{Theorem}[section]
\newtheorem{cor}[thm]{Corollary}
\newtheorem{lem}[thm]{Lemma}

\newtheorem{exam}[thm]{Example}
\numberwithin{equation}{section}

\begin{document}

\title{Group invertibility of the sum in rings and its applications}

\author{Huanyin Chen}
\author{Dayong Liu}
\author{Marjan Sheibani$^*$}
\address{
School of Mathematics\\ Hangzhou Normal University\\ Hangzhou, China}
\email{<huanyinchen@aliyun.com>}
\address{
College of Science \\ Central South University of Forestry and Technology, China}
\email{<liudy@csuft.edu.cn>}
\address{Farzanegan Campus, Semnan University, Semnan, Iran}
\email{<m.sheibani@semnan.ac.ir>}


\begin{abstract} We present new additive results for the group invertibility in a ring.
Then we apply our results to block operator matrices over Banach spaces and derive the existence of group inverses of $2\times 2$ block operator matrices. These
generalize many known results, e.g., Benitez, Liu and Zhu(Linear Multilinear Algebra, {\bf 59}(2011), 279--289) and Zhou, Chen and Zhu(Comm. Algebra, {\bf 48}(2020),676-690).
\end{abstract}

\thanks{Corresponding author: Marjan Sheibani}

\maketitle

\section{Introduction}

Let $R$ be a ring with an identity. An element $a\in R$ has Drazin inverse if there exists $x\in R$ such that $x=xax, ax=xa, a^n=a^{n+1}x$ for
some $n\in {\Bbb N}.$ Such $x$ is unique, if it exists, and we denote it by $a^D$. The smallest $n$ is called the Drazin index of $a$.
If $a$ has Drazin index $1$, $a$ is said to have group inverse $x$, and denote its group inverse by $a^{\#}$. As is well known, a square complex matrix $A$ has group inverse if and only if $rank(A)=rank(A^2)$. The group invertibility in a ring is attractive. It has interesting applications of resistance distances to the bipartiteness of graphs (see~\cite{SW}). Many authors have studied group invertibility from many different views, e.g., ~\cite{B2,BZ,C1,CE,Ca,MD2,M,P,ZM,Z2}. It was also extensively investigated under the concept "strongly regularity" in ring theory (see~\cite{CH}).

In ~\cite[Theorem 2.1]{B}, Benitez et al. studied the group inverse of the sum of two group invertible elements $a$ and $b$ in an algebra
under the condition $ab=0$. The group inverse of $P+Q$ of two group invertible complex matrices $P$ and $Q$ was investigated in~\cite[Theorem 2.3]{L} under the condition $PQQ^{\#}=QPP^{\#}$. Zhou et al. extended the preceding result to a Dedekind-finite ring in which $2$ is invertible (see~\cite[Theorem 3.1]{Z2}).
These motivate us to explore the group inverse of the sum in a general setting.

Let $a\in R^D$. The element $a^{\pi}=1-aa^{D}$ is called the spectral idempotent of $a$. In Section 2, we present new additive results for the group invertibility
by means of spectral idempotents in a ring. Let $a,b,ba^{\pi},ab^{\pi}\in R^{\#}$, $aba^{\pi}=0$ and $bab^{\pi}=0$. We prove that $a+b\in R^{\#}$ if and only if $aa^{\#}b+bb^{\#}a\in R^{\#}, a^{\pi}b^{\pi}a=0$ and $b^{\pi}a^{\pi}b=0$. The preceding known additive properties of the group invertibility are thereby extended to the wider case.

In Section 3, we further investigate the additive properties of the group invertibility under certain commutative-like conditions. Let $a,b\in R^{\#}$. If $a^2b=aba, b^2a=bab$, we prove that $a+b\in R^{\#}$ if and only if $1+a^{\#}b\in R^{\#}$ and $b^{\pi}a^{\pi}b=0$.

Let $X$ and $Y$ be Banach spaces, let $\mathcal{L}(X,Y)$ denote the set of all bounded linear operators from $X$ to $Y$ and $\mathcal{L}(X)$ denote the set of all bounded linear operators from $X$ to itself. The aim of the final section is to explore the the group invertibility of a block operator matrix $$M=\left(
\begin{array}{cc}
A&B\\
C&D
\end{array}
\right)~~~~~~~~~~(*)$$  where $A\in \mathcal{L}(X),B\in \mathcal{L}(X,Y),C\in \mathcal{L}(Y,X)$ and $D\in \mathcal{L}(Y)$. Here, $M$ is a bounded linear operator on $X\oplus Y$. This problem is quite complicated and was studied by many authors (see~~\cite{Ca,M,P,ZM}). In ~\cite{B}, Benitez studied the group inverse of $2\times 2$ block operator matrix $M$ under certain conditions. As applications of our results, we provide many new conditions under which $M$ has group inverse. These extend~\cite[Theorem 3.4-3.7]{B} to the general setting.

Throughout the paper, all rings are associative with an identity. Let $p\in R$ be an idempotent, and let $x\in R$. Then we write $x=pxp+px(1-p)+(1-p)xp+(1-p)x(1-p),$
and induce a Pierce representation given by the matrix
$$x=\left(\begin{array}{cc}
pxp&px(1-p)\\
(1-p)xp&(1-p)x(1-p)
\end{array}
\right)_p.$$

\section{additive properties}

The purpose of this section is to establish new additive results for group inverses in a ring. We begin with

\begin{lem} Let $p\in R$ be an idempotent and $x=\left(
\begin{array}{cc}
a&0\\
c&d
\end{array}
\right)_p.$ If $a,d\in R^{\#}$, then $x\in R^{\#}$ if and only if $d^{\pi}ca^{\pi}=0.$ In this case, $x^{\#}=\left(
\begin{array}{cc}
a^{\#}&0\\
u&d^{\#}
\end{array}
\right)_p,$ where $$u=d^{\pi}c(a^{\#})^2+(d^{\#})^2ca^{\pi}-d^{\#}ca^{\#}.$$\end{lem}
\begin{proof} We obtain the result as in the proof of~\cite[Theorem 2.1]{MD}.\end{proof}

\begin{lem} Let $b\in R^{\#}$ and $p\in R$ be an idempotent. If $pbp^{\pi}=0$ and $bp^{\pi}\in R^{\#}$, then $(bp^{\pi})^{\#}=b^{\#}p^{\pi}$ and $(pb)^{\#}=pb^{\#}$.\end{lem}
\begin{proof} Since $(1-p)bp=0$, we have $b=\left(
\begin{array}{cc}
pb&0\\
p^{\pi}bp&bp^{\pi}
\end{array}
\right)_p.$ Since $b,bp^{\pi}\in R^{\#}$, it follows by Lemma 2.1 that $pb\in R^{\#}$.
In this case, we have $$b^{\#}=\left(
\begin{array}{cc}
(pb)^{\#}&0\\
*&(bp^{\pi})^{\#}
\end{array}
\right)_p.$$ Therefore
$(pb)^{\#}=pb^{\#}$ and $(bp^{\pi})^{\#}=b^{\#}p^{\pi}$, as asserted.\end{proof}

\begin{thm} Let $a,b,ba^{\pi}\in R^{\#}$ and $aba^{\pi}=0$. Then $a+b\in R^{\#}$ if and only if $a(1+a^{\#}b)\in R^{\#}$ and $b^{\pi}a^{\pi}b=0$.\end{thm}
\begin{proof} Let $p=aa^{\#}$. Since $aba^{\pi}=0$, we see that $aa^{\#}ba^{\pi}=0$. Then $$a=\left(
\begin{array}{cc}
a_1&0\\
0&0
\end{array}
\right)_p, b=\left(
\begin{array}{cc}
b_1&0\\
b_3&b_4
\end{array}
\right)_p.$$ Here $a_1=aa^{\#}aaa^{\#}=a, b_3=a^{\pi}baa^{\#}, b_4=a^{\pi}ba^{\pi}=ba^{\pi}$.
Then $$a+b=\left(
\begin{array}{cc}
a_1+b_1&0\\
b_3&b_4
\end{array}
\right)_p.$$

$\Longrightarrow$ By hypothesis, $b_4=ba^{\pi}\in R^{\#}$. Since $aba^{\pi}=0$, it follows by
Lemma 2.2 that $b_4^{\#}=b^{\#}a^{\pi}$. Clearly, $aa^{\#}(a+b)a^{\pi}=0$ and $(a+b)a^{\pi}=ba^{\pi}\in R^{\#}$. Since $a+b\in R^{\#}$, it follows by Lemma 2.2 that $aa^{\#}(a+b)\in R^{\#}$ and $[aa^{\#}(a+b)]^{\#}=aa^{\#}(a+b)^{\#}$. Therefore $aa^{\#}(a+b)=a_1+b_1$ has group inverse. In view of Lemma 2.1, we have
$$\begin{array}{lll}
(a+b)^{\#}&=&\left(
\begin{array}{cc}
w^{\#}&0\\
u&b^{\#}a^{\pi}
\end{array}
\right)_p,
\end{array}$$ where $w=aa^{\#}(a+b)$ and $$\begin{array}{lll}
u&=&(b_4^{\#})^2b_3w^{\pi}+b_4^{\pi}b_3(w^{\#})^2-b_4^{\#}b_3w^{\#}\\
&=&b^{\#}a^{\pi}b^{\#}a^{\pi}baa^{\#}w^{\pi}+b^{\pi}a^{\pi}b(w^{\#})^2-b^{\#}a^{\pi}bw^{\#}.
\end{array}$$
We check that
$$\begin{array}{ll}
&(a+b)^{\pi}(a+b)\\
=&\left(
\begin{array}{cc}
w^{\pi}&0\\
-bw^{\#}-ba^{\pi}u&b^{\pi}a^{\pi}
\end{array}
\right)_p\left(
\begin{array}{cc}
w&0\\
a^{\pi}baa^{\#}&ba^{\pi}
\end{array}
\right)_p\\
=&\left(
\begin{array}{cc}
w^{\pi}w&0\\
-(bw^{\#}+ba^{\pi}u)w+b^{\pi}a^{\pi}baa^{\#}&b^{\pi}a^{\pi}ba^{\pi}\\
\end{array}
\right)_p\\
=&0.
\end{array}$$ Hence, $-(bw^{\#}+ba^{\pi}u)w+b^{\pi}a^{\pi}baa^{\#}=0$. Thus $$b^{\pi}a^{\pi}baa^{\#}=b^{\pi}[-(bw^{\#}+ba^{\pi}u)w+b^{\pi}a^{\pi}baa^{\#}]=0.$$
On the other hand, $b^{\pi}a^{\pi}ba^{\pi}=b^{\pi}ba^{\pi}-b^{\pi}a^{\#}aba^{\pi}=0$. Therefore $b^{\pi}a^{\pi}b=b^{\pi}a^{\pi}b[aa^{\#}+a^{\pi}]=0$, as required.

$\Longleftarrow$ We easily check that
$$\begin{array}{lll}
b_4^{\pi}b_3(a_1+b_1)^{\pi}&=&(1-b^{\#}a^{\pi}ba^{\pi})b_3(a_1+b_1)^{\pi}\\
&=&(a^{\pi}-b^{\#}ba^{\pi})b_3(a_1+b_1)^{\pi}\\
&=&[b^{\pi}a^{\pi}b]aa^{\#}(a_1+b_1)^{\pi}\\
&=&0.
\end{array}$$ According to Lemma 2.1, $a+b\in R^{\#}$. Moreover, we have
$$\begin{array}{lll}
(a+b)^{\#}&=&(a+b)^D\\
&=&w^{\#}+u+b^{\#}a^{\pi}\\
&=&w^{\#}+b^{\#}a^{\pi}b^{\#}a^{\pi}baa^{\#}w^{\pi}-b^{\#}a^{\pi}bw^{\#}+b^{\#}a^{\pi}\\
&=&[a(1+a^{\#}b)]^{\#}+b^{\#}a^{\pi}b^{\#}a^{\pi}baa^{\#}[a(1+a^{\#}b)]^{\pi}\\
&-&b^{\#}a^{\pi}b[a(1+a^{\#}b)]^{\#}+b^{\#}a^{\pi},
\end{array}$$ as asserted.\end{proof}

\begin{cor} Let $a,b,a^{\pi}b\in R^{\#}$ and $a^{\pi}ba=0$. Then $a+b\in R^{\#}$ if and only if $(1+ba^{\#})a\in R^{\#}$ and $ba^{\pi}b^{\pi}=0$. In this case, $$\begin{array}{lll}
(a+b)^{\#}&=&[(1+ba^{\#})a]^{\#}+[(1+ba^{\#})a]^{\pi}aa^{\#}ba^{\pi}b^{\#}a^{\pi}b^{\#}\\
&-&[(1+ba^{\#})a]^{\#}ba^{\pi}b^{\#}+a^{\pi}b^{\#}.\\
\end{array}$$\end{cor}
\begin{proof} Since $(R,\cdot)$ is a ring, $(R,*)$ is a ring with the multiplication $a*b=b\cdot a$, i.e., it is the opposite ring. Then we complete the proof by applying Theorem 2.3 to the ring $(R,*)$.\end{proof}

Now our first major result is demonstrated as follows.

\begin{thm} Let $a,b,ba^{\pi},ab^{\pi}\in R^{\#}$, $aba^{\pi}=0$ and $bab^{\pi}=0$. Then $a+b\in R^{\#}$ if and only if $aa^{\#}b+bb^{\#}a\in R^{\#},
a^{\pi}b^{\pi}a=0$ and $b^{\pi}a^{\pi}b=0$.\end{thm}
\begin{proof} In view of Theorem 2.3, $a+b\in R^{\#}$ if and only if $a+aa^{\#}b=a(1+a^{\#}a)\in R^{\#}$ and $b^{\pi}a^{\pi}b=0$.
Since $aa^{\#}ba^{\pi}=0$ and $ba^{\pi}\in R^D$, it follows by Lemma 2.2 that
$(aa^{\#}b)^{\#}=aa^{\#}b^{\#}$. Let $q=[aa^{\#}b][aa^{\#}b]^{\#}$. Then $q=aa^{\#}bb^{\#}$.
Obviously, we have $qa(1-q)=aa^{\#}bb^{\#}a(1-aa^{\#}bb^{\#})=aa^{\#}bb^{\#}ab^{\pi}=0$. Then $$a=\left(
\begin{array}{cc}
a_1&0\\
a_3&a_4
\end{array}
\right)_q, aa^{\#}b=\left(
\begin{array}{cc}
b_1&0\\
0&0
\end{array}
\right)_q.$$ Here $a_1=qaq, a_3=q^{\pi}aq, a_4=q^{\pi}aq^{\pi}$.
Then $$a+aa^{\#}b=\left(
\begin{array}{cc}
a_1+b_1&0\\
a_3&a_4
\end{array}
\right)_q.$$ Here, we have
$$\begin{array}{rll}
a_4&=&(1-aa^{\#}bb^{\#})a(1-aa^{\#}bb^{\#})\\
&=&(1-aa^{\#}bb^{\#})ab^{\pi}\\
&=&ab^{\pi}\\
&\in&R^{\#}.
\end{array}$$

$\Longrightarrow$ Since $aba^{\pi}=0$, it follows by Theorem 2.3 that $a^{\pi}b^{\pi}a=0$. Hence,
$$\begin{array}{lll}
a_1+b_1&=&qaq+qbq\\
&=&aa^{\#}bb^{\#}abb^{\#}+aa^{\#}b^2b^{\#}aa^{\#}bb^{\#}\\
&=&aa^{\#}bb^{\#}a+aa^{\#}b\\
&=&bb^{\#}a-a^{\pi}bb^{\#}a+aa^{\#}b\\
&=&aa^{\#}b+bb^{\#}a.
\end{array}$$
Since $bab^{\pi}=0$, it follows by
Lemma 2.2 that $a_4^{\#}=a^{\#}b^{\pi}$ and $a_4^{\pi}=q^{\pi}-aa^{\#}b^{\pi}=a^{\pi}$. Moreover, we have
$a_3=q^{\pi}aq=(1-aa^{\#}bb^{\#})a^2a^{\#}bb^{\#}=abb^{\#}-aa^{\#}bb^{\#}a$. Hence $a^{\pi}a_3=0$.

By hypothesis, $ba^{\pi}\in R^{\#}$. Since $a+b\in R^{\#}$, it follows by Theorem 2.3 that
$aa^{\#}(a+b)=a+aa^{\#}b=a(1+a^{\#}b)$ has group inverse. Set $z=a_1+b_1$. Then we have
$$\begin{array}{lll}
(a+aa^{\#}b)^{\#}&=&\left(
\begin{array}{cc}
(a_1+b_1)^{\#}&\\
v&a_4^{\#}
\end{array}
\right)_q\\
&=&\left(
\begin{array}{cc}
z^{\#}&0\\
v&a^{\#}b^{\pi}
\end{array}
\right)_q,
\end{array},$$ where $$\begin{array}{lll}
v&=&(a_4^{\#})^2a_3z^{\pi}+a_4^{\pi}a_3(z^{\#})^2-a_4^{\#}a_3z^{\#}\\
&=&[a^{\#}b^{\pi}]^2a_3z^{\pi}+a^{\pi}a_3(z^{\#})^2-a^{\#}b^{\pi}a_3z^{\#}\\
&=&[a^{\#}b^{\pi}]^2[abb^{\#}-aa^{\#}bb^{\#}a]z^{\pi}-a^{\#}b^{\pi}[abb^{\#}-aa^{\#}bb^{\#}a]z^{\#}.
\end{array}$$ We easily verify that $a_1+b_1\in R^{\#}$, and so $aa^{\#}b+bb^{\#}a\in R^{\#}.$

$\Longleftarrow$ Clearly, $a_4\in R^{\#}$ and $$\begin{array}{lll}
a_1+b_2&=&q(a+aa^{\#}b)q\\
&=&aa^{\#}bb^{\#}(a+aa^{\#}b)\\
&=&bb^{\#}a-a^{\pi}bb^{\#}a+aa^{\#}bb^{\#}aa^{\#}b\\
&=&bb^{\#}a+aa^{\#}b-aa^{\#}bb^{\#}a^{\pi}b\\
&=&bb^{\#}a+aa^{\#}b-aa^{\#}(1-b^{\pi})a^{\pi}b\\
&=&aa^{\#}b+bb^{\#}a\\
&\in&R^{\#}.
\end{array}$$
Moreover, we check that
$$\begin{array}{lll}
a_4^{\pi}a_3(a_1+b_1)^{\pi}&=&a^{\pi}b^{\pi}abb^{\#}(a_1+b_1)^{\pi}\\
&=&0.
\end{array}$$ According to Lemma 2.1, $$\begin{array}{lll}
a(1+a^{\#}b)&=&a+aa^{\#}b\\
&\in&R^{\#}.
\end{array}$$ Moreover,
$$\begin{array}{rll}
[a(1+a^{\#}b)]^{\#}&=&[aa^{\#}b+bb^{\#}a]^{\#}+a^{\#}b^{\pi}\\
&+&[a^{\#}b^{\pi}]^2[abb^{\#}-aa^{\#}bb^{\#}a][aa^{\#}b+bb^{\#}a]^{\pi}\\
&-&a^{\#}b^{\pi}[abb^{\#}-aa^{\#}bb^{\#}a][aa^{\#}b+bb^{\#}a]^{\#}.
\end{array}$$ This completes the proof.\end{proof}

In ~\cite[Theorem 3.1]{Z2}, Zhou et al. investigated the group inverse of $a+b$ under the condition $aa^{\#}b=bb^{\#}a$ in a Dedekind ring with $2$ is invertible. We now extend this result to a general setting.

\begin{cor} Let $a,b,aa^{\#}b\in R^{\#}$ and $aa^{\#}b=bb^{\#}a$. Then $a+b\in R^{\#}$ if and only if $2aa^{\#}b\in R^{\#}$.\end{cor}
\begin{proof} Since $aa^{\#}b=bb^{\#}a$, we check that $aba^{\pi}=a(aa^{\#}b)a^{\pi}=a(bb^{\#}a)a^{\pi}=0$.
Similarly, we have $bab^{\pi}=b^{\#}(b^2a)b^{\pi}=b^{\#}ba(bb^{\pi})=0$. Moreover, we have
$$\begin{array}{c}
a^{\pi}b^{\pi}a=a^{\pi}(1-bb^{\#})a=a^{\pi}bb^{\#}a=a^{\pi}aa^{\#}b=0,\\
b^{\pi}a^{\pi}b=b^{\pi}(1-aa^{\#})b=b^{\pi}aa^{\#}b=b^{\pi}bb^{\#}a=0.
\end{array}$$ Since $aa^{\#}b\in R^{\#}$ and $aa^{\#}ba^{\pi}=0$, similarly to Lemma 2.2, $ba^{\pi}\in R^{\#}$ and $[ba^{\pi}]^{\#}=b^{\#}a^{\pi}$. Likewise, $ab^{\pi}\in R^{\#}$. Therefore we complete the proof by Theorem 2.5.\end{proof}

Applying Corollary 2.6 to the opposite ring $(R,*)$, we dually derive

\begin{cor} Let $a,b,abb^{\#}\in R^{\#}$ and $abb^{\#}=baa^{\#}$. Then $a+b\in R^{\#}$ if and only if $2abb^{\#}\in R^{\#}$.\end{cor}

\section{communicative-like conditions}

The aim of this section is to investigate the group inverse in a ring under some commutative-like conditions. We now prove:

\begin{lem} Let $a,b\in R^{\#}$. If $a^2b=aba$ and $b^2a=bab$, then $ab\in R^{\#}$ and $$(ab)^{\#}=a^{\#}b^{\#}.$$
\end{lem}\begin{proof} Since $a,b\in R^{\#}$, we have $a,b\in R^D$. In view of~\cite[Theorem 3.1]{Z}, $ab\in R^D$ and $(ab)^D=a^Db^D=a^{\#}b^{\#}$.
One easily checks that$$(ab)^2(ab)^D=ababa^{\#}b^{\#}=(a^2a^{\#})(b^2b^{\#})=ab.$$ Therefore $(ab)^{\#}=a^{\#}b^{\#},$ as desired.\end{proof}

\begin{thm} Let $a,b\in R^{\#}$. If $a^2b=aba, b^2a=bab$, then $a+b\in R^{\#}$ if and only if $aa^{\#}b+bb^{\#}a\in R^{\#}, a^{\pi}b^{\pi}a=0, b^{\pi}a^{\pi}b=0$.\end{thm}
\begin{proof} Since $a^2b=aba$, we have $aba^{\pi}=a^{\#}(a^2b)a^{\pi}=a^{\#}(aba)a^{\pi}=0$. It follows from $b^2a=bab$ that $bab^{\pi}=0$.
Since $a(ab)=a^2b=aba=(ab)a$, we have $a^{\#}(ab)=a^D(ab)=(ab)a^D=(ab)a^{\#}$ by~\cite[Theorem 2.2]{D}; hence,
$(aa^{\#})^2b=a^{\#}(ab)=(ab)a^{\#}=aa^{\#}baa^{\#}$. Also we have $b^2aa^{\#}=(bab)a^{\#}=baa^{\#}b$.
In view of Lemma 3.1, $(aa^{\#}b)^{\#}=aa^{\#}b^{\#}.$ Since $aa^{\#}ba^{\pi}=0$, analogously to Lemma 2.2, $ba^{\pi}\in R^{\#}$ and $(ba^{\pi})^{\#}=b^{\#}a^{\pi}$.
Similarly, we show that $ab^{\pi}\in R^{\#}$. Therefore we complete the proof by Theorem 2.5.\end{proof}

\begin{cor} Let $a,b\in R^{\#}$. If $ab=ba$, then $a+b\in R^{\#}$ if and only if $aa^{\#}b+bb^{\#}a\in R^{\#}$.\end{cor}
\begin{proof} Since $ab=ba$, we have $a^2b=aba$ and $b^2a=ba$. The result follows by Theorem 3.2.\end{proof}

We now illustrates Theorem 3.2 by the following example:

\begin{exam} Let $a=
\left(
\begin{array}{cccc}
1&0&0&0\\
0&1&0&0\\
0&0&0&1\\
0&0&0&0
\end{array}
\right), b=\left(
\begin{array}{cccc}
\frac{1}{3}&0&0&0\\
0&0&0&0\\
0&0&0&0\\
0&0&0&\frac{1}{3}
\end{array}
\right)\in {\Bbb C}^{4\times 4}$. Then $a^2b=aba, b^2a=bab$. But $ab\neq ba$. We compute that
$$a^{\#}=\left(
\begin{array}{cccc}
1&0&0&0\\
0&1&0&0\\
0&0&0&0\\
0&0&0&0
\end{array}
\right), b^{\#}=\left(
\begin{array}{cccc}
3&0&0&0\\
0&0&0&0\\
0&0&0&0\\
0&0&0&3
\end{array}
\right).$$ In this case,
$$(aa^{\#}b+bb^{\#}a)^{\#}=\left(
\begin{array}{cccc}
\frac{3}{4}&0&0&0\\
0&0&0&0\\
0&0&0&0\\
0&0&0&0
\end{array}
\right),(a+b)^{\#}=\left(
\begin{array}{cccc}
\frac{3}{4}&0&0&0\\
0&1&0&0\\
0&0&0&0\\
0&0&0&3
\end{array}
\right).$$\end{exam}

We are now ready to prove:

\begin{thm} Let $a,b\in R^{\#}$. If $a^2b=aba, b^2a=bab$, then $a+b\in R^{\#}$ if and only if $1+a^{\#}b\in R^{\#}$ and $b^{\pi}a^{\pi}b=0$.\end{thm}
\begin{proof} $\Longrightarrow$ Since $a^2b=aba$, we see that $aba^{\pi}=a^{\#}(a^2b)a^{\pi}=a^{\#}abaa^{\pi}=0$. As in the proof of Theorem 3.2, we show that $ba^{\pi}$ has group inverse. By virtue of Theorem 2.3, $b^{\pi}a^{\pi}b=0$. Write $1+a^{\#}b=a^{\pi}+a^{\#}(a+b)$. One easily checks that
$$\begin{array}{rll}
(a^{\#})^2(a+b)&=&a^{\#}+(a^{\#})^2b\\
&=&a^{\#}+(a^{\#})^3(ab)\\
&=&a(a^{\#})^2+(a^{\#})^2(ab)a^{\#}\\
&=&a^{\#}aa^{\#}+a^{\#}ba^{\#}\\
&=&a^{\#}(a+b)a^{\#}.
\end{array}$$ Likewise, we have $(a+b)^2a^{\#}=(a+b)a^{\#}(a+b).$ In light of Lemma 3.1, $[a^{\#}(a+b)]^{\#}=a(a+b)^{\#}$.
Clearly, $a^{\pi}[a^{\#}(a+b)]=0$. Also we have $[a^{\#}(a+b)]a^{\pi}= a^{\#}ba^{\pi}=(a^{\#})^3(a^2b)a^{\pi}=(a^{\#})^3(aba)a^{\pi}=0$.
According to Corollary 2.6, $a^{\pi}+a^{\#}(a+b)\in R^{\#}$. Then $1+a^{\#}b\in R^{\#}$, as desired.

$\Longleftarrow$ Clearly, $aa^{\#}(a+b)=a(1+a^{\#}b)$. By hypothesis, we have $$a(1+a^{\#}b)=a+aa^{\#}b=a+aa^{\#}b=a+(a^{\#})^2aba=(1+a^{\#}b)a.$$ By virtue of Lemma 3.1, $[aa^{\#}(a+b)]^{\#}=a^{\#}(1+a^{\#}b)^{\#}$. As in the proof of Theorem 3.2, $ba^{\pi}\in R^{\#}$. Since $b^{\pi}a^{\pi}b=0$, we complete the proof by Theorem 2.3.\end{proof}

\begin{cor} Let $a,b\in R^{\#}$. If $ab^2=bab, ba^2=aba$, then $a+b\in R^{\#}$ if and only if $1+ba^{\#}\in R^{\#}$ and $ba^{\pi}b^{\pi}=0$.\end{cor}
\begin{proof} Since $(R,\cdot)$ is a ring, $(R,*)$ is a ring with the multiplication $a*b=b\cdot a$.
Then we complete the proof by applying Theorem 3.5 to the opposite ring $(R,*)$.\end{proof}

\section{applications}

We will investigate the group invertibility of a block operator matrix $M$ as in $(*)$. The main idea is to split $M$ as the sum of two special block operator matrices. Then we apply our additive results and find new conditions on bounded linear operators $A,B,C$ and $D$ under which $M$ is group invertible.

\begin{thm} Let $A,D$ and $BC$ have group inverses. If $B(CB)^{\pi}=0, C(BC)^{\pi}=0, ABD^{\pi}=0$ and $DCA^{\pi}=0$, then $M$ has group inverse if and only if
\begin{enumerate}
\item [(1)] $A^{\pi}(BC)^{\pi}A=0, D^{\pi}(CB)^{\pi}D=0$;
\item [(2)] $\left(
  \begin{array}{cc}
   A&AA^{\#}B \\
   DD^{\#}C &D
  \end{array}
\right)$ has group inverse.\end{enumerate}\end{thm}
\begin{proof} Write $M=P+Q$, where $$P=\left(
  \begin{array}{cc}
    A & 0 \\
    0 & D
  \end{array}
\right), Q=\left(
  \begin{array}{cc}
   0 & B \\
   C & 0
  \end{array}
\right).$$ Then $$P^{\#}=\left(
  \begin{array}{cc}
    A^{\#} & 0 \\
    0 & D^{\#}
  \end{array}
\right), P^{\pi}=\left(
  \begin{array}{cc}
    A^{\pi} & 0 \\
    0 & D^{\pi}
  \end{array}
\right).$$ Since $BC$ has group inverse, we have $$(Q^2)^D=\left(
  \begin{array}{cc}
   BC & 0 \\
   0 & CB
  \end{array}
\right)^D=\left(
  \begin{array}{cc}
   (BC)^{\#} & 0 \\
   0 & (CB)^{\#}
  \end{array}
\right).$$ Since $\mathcal{L}(X\oplus Y)$ is a Banach algebra, it follows by ~\cite[Theorem 2.1]{JW} that $Q$ has Drazin inverse and $$Q^D=Q(Q^2)^D=\left(
  \begin{array}{cc}
  0& B(CB)^{\#} \\
  C(BC)^{\#}&0
  \end{array}
\right).$$ Moreover, we have $$Q^{\pi}=\left(
  \begin{array}{cc}
  (BC)^{\pi}&0\\
  0&(CB)^{\pi}
  \end{array}
\right).$$
Hence, $$QQ^D=Q^DQ, Q^D=Q^DQQ^D.$$ It is easy to verify that
$$\begin{array}{lll}
QQ^{\pi}&=&\left(
  \begin{array}{cc}
   0 & B \\
   C & 0
  \end{array}
\right)\left(
  \begin{array}{cc}
   (BC)^{\pi} & 0 \\
  0 & (CB)^{\pi}
  \end{array}
\right)\\
&=&\left(
  \begin{array}{cc}
  0&B(CB)^{\pi}\\
  C(BC)^{\pi} & 0
  \end{array}
\right)\\
&=&0;
\end{array}$$ whence, $Q=Q^2Q^D$. Then $Q$ has group inverse and $$Q^{\#}=\left(
  \begin{array}{cc}
  0& B(CB)^{\#} \\
  C(BC)^{\#}&0
  \end{array}
\right).$$
Therefore we compute that
$$\begin{array}{lll}
P^{\pi}Q^{\pi}P&=&\left(
  \begin{array}{cc}
  A^{\pi}(BC)^{\pi}&0\\
  0&D^{\pi}(CB)^{\pi}
  \end{array}
\right)\left(
  \begin{array}{cc}
    A & 0 \\
    0 & D
  \end{array}
\right)\\
&=&\left(
  \begin{array}{cc}
  A^{\pi}(BC)^{\pi}A&0\\
  0&D^{\pi}(CB)^{\pi}D
  \end{array}
\right)\\
&=&0.
\end{array}$$
We easily check that
$$\begin{array}{rll}
PQP^{\pi}&=&\left(
  \begin{array}{cc}
   0&ABD^{\pi}\\
   DCA^{\pi}&0
  \end{array}
\right),\\
PP^{\#}(P+Q)&=&\left(
  \begin{array}{cc}
   A&AA^{\#}B \\
   DD^{\#}C &D
  \end{array}
\right).
\end{array}$$ Therefore we complete the proof by Theorem 2.3.\end{proof}

\begin{cor} Let $A,D$ and $BC$ have group inverses. If $AB=0,B(CB)^{\pi}=0, C(BC)^{\pi}=0$ and $DCA^{\pi}=0$, then $M$ has group inverse if and only if $BCA=0$ and $D^{\pi}(CB)^{\pi}D=0$.\end{cor}
\begin{proof} Clearly, $D^{\pi}(DD^{\#}C)A^{\pi}=0$. As in the proof of~\cite[Theorem 3.2]{B},
$\left(
  \begin{array}{cc}
   A&AA^{\#}B \\
   DD^{\#}C &D
  \end{array}
\right)$ has group inverse. This completes the proof by Theorem 4.1.\end{proof}

\begin{cor} Let $A,D$ and $BC$ have group inverses. If $DC=0, B(CB)^{\pi}=0, C(BC)^{\pi}=0$ and $ABD^{\pi}=0$, then $M$ has group inverse if and only if $A^{\pi}(BC)^{\pi}A=0$ and $CBD=0$.\end{cor}
\begin{proof} Obviously, $A^{\pi}(AA^{\#}B)D^{\pi}=0$. Similarly to~\cite[Theorem 3.1]{B}, $\left(
  \begin{array}{cc}
   A&AA^{\#}B \\
   DD^{\#}C &D
  \end{array}
\right)$ has group inverse. There we obtain the result by Theorem 4.1.\end{proof}

\begin{cor} Let $A\in {\Bbb C}^{m\times m}$ and $D\in {\Bbb C}^{n\times n}$ have group inverses. If $r(B)=r(C)=r(BC)=r(CB)$, then \end{cor}
\begin{enumerate}
\item [(1)] If $AB=0$ and $DCA^{\pi}=0$, then $M$ has group inverse if and only if $BCA=0$ and $D^{\pi}(CB)^{\pi}D=0$.
\item [(2)] If $DC=0$ and $ABD^{\pi}=0$, then $M$ has group inverse if and only if $A^{\pi}(BC)^{\pi}A=0$ and $CBD=0$.
\end{enumerate}
\begin{proof} Since $r(B)=r(C)=r(BC)=r(CB)$, by virtue of~\cite[Lemma 2.3]{BZ}, $BC$ and $CB$ have group inverses.
Let $N=\left(
  \begin{array}{cc}
   0&B \\
   C&0
  \end{array}
\right)$. Then $N^2=\left(
  \begin{array}{cc}
   BC&0 \\
   0&CB
  \end{array}
\right).$ Hence, $rank(N^2)=rank(BC)+rank(CB)=rank(B)+rank(C)=rank(N)$.
Therefore $N$ has group inverse. This implies that $$NN^{\pi}=\left(
  \begin{array}{cc}
   0&B(CB)^{\pi} \\
   C(BC)^{\pi}&0
  \end{array}
\right)=0,$$ and so $B(CB)^{\pi}=0, C(BC)^{\pi}=0.$ Therefore we complete the proof by Corollary 4.2 and Corollary 4.3.\end{proof}

We now ready to prove:

\begin{thm} Let $A$ and $D$ have group inverses. If $A^{\pi}B=0, D^{\pi}C=0, BCA^{\pi}=0, ABC=0$ and $BD=BCA^{\#}B$, then $M$
has group inverse.\end{thm}
\begin{proof} Write $M=P+Q$, where $$P=\left(
  \begin{array}{cc}
    A&B\\
    0&0
  \end{array}
\right), Q=\left(
  \begin{array}{cc}
    0&0\\
    C&D
  \end{array}
\right).$$
Since $A^{\pi}B=0, D^{\pi}C=0$, we easily see that $P$ and $Q$ have group inverses. Moreover, we have $$\begin{array}{c}
P^{\#}=\left(
  \begin{array}{cc}
  A^{\#}&(A^{\#})^2B\\
  0&0
  \end{array}
  \right), P^{\pi}=\left(
  \begin{array}{cc}
     A^{\pi}&-A^{\#}B\\
    0&I_n
  \end{array}
\right);\\
Q^{\#}=\left(
  \begin{array}{cc}
    0&0\\
    (D^{\#})^2C&D^{\#}
  \end{array}
\right),Q^{\pi}=\left(
  \begin{array}{cc}
    I_m&0\\
    -D^{\#}C&D^{\pi}
  \end{array}
\right).
\end{array}$$
We compute that
$$\begin{array}{lll}
PP^{\#}Q&=&\left(
  \begin{array}{cc}
     AA^{\#}&A^{\#}B\\
    0&0
  \end{array}
\right)\left(
  \begin{array}{cc}
    0&0\\
    C&D
  \end{array}
\right)\\
&=&\left(
  \begin{array}{cc}
    A^{\#}BC&A^{\#}BD\\
    0&0
  \end{array}
\right)\\
&=&0.
\end{array}$$
Then $$[QP^{\pi}]^D=Q[(P^{\pi}Q)^D]^2P^{\pi}=Q(Q^D)^2P^{\pi}=Q^DP^{\pi}.$$ We check that
$[QP^{\pi}]^2[QP^{\pi}]^D=Q^2Q^DP^{\pi}=QP^{\pi}.$ Therefore $QP^{\pi}$ has group inverse.
Moreover, we check that
$$\begin{array}{rll}
PQP^{\pi}&=&\left(
  \begin{array}{cc}
    A&B\\
    0&0
  \end{array}
\right)\left(
  \begin{array}{cc}
    0&0\\
    C&D
  \end{array}
\right)\left(
  \begin{array}{cc}
     A^{\pi}&-A^{\#}B\\
    0&I_n
  \end{array}
\right)\\
&=&\left(
  \begin{array}{cc}
     BCA^{\pi}&BD-BCA^{\#}B\\
    0&0
  \end{array}
\right)\\
&=&0,\\
Q^{\pi}P^{\pi}Q&=&\left(
  \begin{array}{cc}
    I&0\\
    -D^{\#}C&D^{\pi}
  \end{array}
\right)\left(
  \begin{array}{cc}
     A^{\pi}&-A^{\#}B\\
    0&I
  \end{array}
\right)\left(
  \begin{array}{cc}
    0&0\\
    C&D
  \end{array}
\right)\\
&=&\left(
  \begin{array}{cc}
    I&0\\
    -D^{\#}C&D^{\pi}
  \end{array}
\right)\left(
  \begin{array}{cc}
   -A^{\#}BC&-A^{\#}BD\\
    C&D
  \end{array}
\right)\\
&=&\left(
  \begin{array}{cc}
   -A^{\#}BC&-A^{\#}BD\\
    D^{\#}CA^{\#}BC&D^{\#}CA^{\#}BD
  \end{array}
\right)\\
&=&0.
\end{array}$$ Further, we have
$$\begin{array}{lll}
PP^{\#}(P+Q)&=&\left(
  \begin{array}{cc}
     AA^{\#}&A^{\#}B\\
    0&0
  \end{array}
\right)\left(
  \begin{array}{cc}
    A&B\\
    C&D
  \end{array}
\right)\\
&=&\left(
  \begin{array}{cc}
    A+A^{\#}BC&AA^{\#}B+A^{\#}BD\\
    0&0
  \end{array}
\right)\\
&=&\left(
  \begin{array}{cc}
    A&AA^{\#}B\\
    0&0
  \end{array}
\right).
\end{array}$$ Since $A^{\pi}(AA^{\#}B)=0$ and $A$ has group inverse, $PP^{\#}(P+Q)$ has group inverse. Therefore we complete the proof by Theorem 2.3.\end{proof}

\begin{cor} (~\cite[Theorem 3.4]{B}) Let $A$ and $D$ have group inverses. If $A^{\pi}B=0, D^{\pi}C=0, BC=0$ and $BD=0$, then $M$ has group inverse.\end{cor}
\begin{proof} This is obvious by Theorem 4.5.\end{proof}

\begin{thm} Let $A,D$ and have group inverses. If $CA^{\pi}=0, BD^{\pi}$ $=0, A^{\pi}BC=0, BCA=0$ and $DC=CA^{\#}BC$, then $M$
has group inverse.\end{thm}
\begin{proof} Write $M=P+Q$, where $$P=\left(
  \begin{array}{cc}
    A&0\\
    C&0
  \end{array}
\right), Q=\left(
  \begin{array}{cc}
    0&B\\
    0&D
  \end{array}
\right).$$
Since $CA^{\pi}=0, BD^{\pi}=0$, $P$ and $Q$ have group inverses. Moreover, we have $$\begin{array}{c}
P^{\#}=\left(
  \begin{array}{cc}
    A^{\#}&0\\
    C(A^{\#})^2&0
  \end{array}
\right), P^{\pi}=\left(
  \begin{array}{cc}
    A^{\pi}&0\\
    -CA^{\#}&I
  \end{array}
\right);\\
Q^{\#}=\left(
  \begin{array}{cc}
    0&B(D^{\#})^2\\
    0&D^{\#}
  \end{array}
\right),Q^{\pi}=\left(
  \begin{array}{cc}
     I&-BD^{\#}\\
    0&D^{\pi}
  \end{array}
\right).
\end{array}$$ We check that
$$\begin{array}{lll}
QPP^{\#}&=&\left(
  \begin{array}{cc}
    0&B\\
    0&D
  \end{array}
\right)\left(
  \begin{array}{cc}
    AA^{\#}&0\\
    CA^{\#}&0
  \end{array}
\right)\\
&=&\left(
  \begin{array}{cc}
    BCA^{\#}&0\\
    DCA^{\#}&0
  \end{array}
\right)\\
&=&0.
\end{array}$$
By using Cline's formula, we have
$$\begin{array}{lll}
[P^{\pi}Q]^D&=&P^{\pi}[(QP^{\pi})^D]^2Q\\
&=&P^{\pi}[Q^D]^2Q\\
&=&P^{\pi}Q^D.
\end{array}$$
Since $[P^{\pi}Q]^2[P^{\pi}Q]^D=P^{\pi}Q$, it follows that $P^{\pi}Q$ has group inverse.
Moreover, we verify that
$$\begin{array}{lll}
P^{\pi}QP&=&\left(
  \begin{array}{cc}
    0&A^{\pi}B\\
    0&D-CA^{\#}B
  \end{array}
\right)\left(
  \begin{array}{cc}
    A&0\\
    C&0
  \end{array}
\right)\\
&=&\left(
  \begin{array}{cc}
    A^{\pi}BC&0\\
    (D-CA^{\#}B)C&0
  \end{array}
\right).
\end{array}$$
Moreover, we compute that
$$\begin{array}{lll}
QP^{\pi}Q^{\pi}&=&\left(
  \begin{array}{cc}
    0&B\\
    0&D
  \end{array}
\right)\left(
  \begin{array}{cc}
    A^{\pi}&0\\
    -CA^{\#}&I
  \end{array}
\right)\left(
  \begin{array}{cc}
     I&-BD^{\#}\\
    0&D^{\pi}
  \end{array}
\right)\\
&=&\left(
  \begin{array}{cc}
    -BCA^{\#}&BCA^{\#}BD^{\#}\\
    -DCA^{\#}&DCA^{\#}BD^{\#}
  \end{array}
\right)\\
&=&0.
\end{array}$$ We also see that
$$\begin{array}{lll}
(I+QP^{\#})P&=&(P+Q)PP^{\#}\\
&=&\left(
  \begin{array}{cc}
    A&B\\
    C&D
  \end{array}
\right)\left(
  \begin{array}{cc}
    AA^{\#}&0\\
    CA^{\#}&0
  \end{array}
\right)\\
&=&\left(
  \begin{array}{cc}
    A&0\\
    CAA^{\#}&0
  \end{array}
\right).
\end{array}$$ Since $[CAA^{\#}]A^{\pi}=0$ and $A$ has group inverse, $(I+QP^{\#})P$ has group inverse. Therefore $M$ has group inverse by Corollary 2.4.\end{proof}

\begin{cor} (~\cite[Theorem 3.6]{B}) Let $A$ and $D$ have group inverses. If $CA^{\pi}=0, BD^{\pi}=0, BC=0$ and $DC=0$, then $M$ has group inverse.\end{cor}
\begin{proof} This is obvious by Theorem 4.7.\end{proof}

\end{document}